\documentclass{article}
\usepackage{graphicx}
\usepackage[utf8]{inputenc}
\usepackage{kotex}
\usepackage{amsfonts}
\usepackage{amsmath}
\usepackage{amsthm}
\usepackage{amssymb}
\usepackage{tikz}
\usepackage[margin=1.0in]{geometry}
\usepackage{caption}
\usepackage{float}
\usepackage{mathtools}

\begin{document}
\nocite{1,2,3,4,5,6,7,8}

\newtheorem{thm}{Theorem}
\newtheorem{cor}{Corollary}
\newtheorem{lemma}{Lemma}

\theoremstyle{remark}
\newtheorem{rmk}{Remark}

\theoremstyle{definition}
\newtheorem{defn}{Definition}
\newtheorem{prop}{Proposition}

\newcommand{\abs}[1]{\left\lvert#1\right\rvert}
\newcommand{\norm}[1]{\left\lVert#1\right\rVert}
\newcommand{\var}[1]{\underset{#1}{\mathrm{var}}}

\title{\large\textbf{A CENTRAL LIMIT THEOREM FOR\\ROSEN CONTINUED FRACTIONS}}
\author{Juno Kim, Kyuhyeon Choi}

\maketitle

\begin{abstract}
We prove a central limit theorem for Birkhoff sums of the Rosen continued fraction algorithm. A Lasota-Yorke bound is obtained for general one-dimensional continued fractions with the bounded variation space, which implies quasi-compactness of the transfer operator. The main result is a direct proof of the existence of a spectral gap, assuming a certain behavior of the transformation when iterated. This condition is explicitly proved for the Rosen system. We conclude via well-known results of A. Broise that the central limit theorem holds.
\end{abstract}

\section{Introduction}

Our main systems of interest are the \textit{Rosen continued fractions}, a generalization of the Euclidean nearest-integer algorithm. For a fixed integer $q\geq 3$, let $\lambda_q=2\cos(\frac{\pi}{q})$ and $I_q=[-\lambda_q/2,\lambda_q/2]$. The Rosen map $T_q:I_q\rightarrow I_q$ is given as:

\begin{equation*}
T_q(x)=\abs{\frac{1}{x}}-\lambda_q \left\lfloor{\abs{\frac{1}{\lambda_q x}}+\frac{1}{2}}\right\rfloor,\;\;\;T(0)=0\,.
\end{equation*}

This corresponds to a continued fraction expansion of the form
\begin{equation}\label{expansion}
x=\frac{\epsilon_1}{a_1\lambda_q+\frac{\epsilon_2}{a_2\lambda_q+\cdots}}\,,
\end{equation}
where $\epsilon_n=\mathrm{sgn}(T^{n-1}(x))$ and $a_n=\left\lfloor{\abs{ 1/\lambda_q T^{n-1}x}+1/2} \right\rfloor$. When $q=3$, we retrieve the nearest-integer algorithm.

Our goal is to show that Birkhoff sums of the Rosen continued fractions are asymptotically Gaussian. In [2], Broise showed such a central limit system for general piecewise expanding maps of the interval under some conditions, including the assumption that 1 is a simple eigenvalue of the corresponding Perron-Frobenius operator. Our paper gives a proof of the existence of a spectral gap for Rosen continued fractions.

We develop our theory for a more general class of 1-dimensional continued fractions, detailed in Section 2. In Section 3, we show a Lasota-Yorke type bound, which is used in Section 4 to show quasi-compactness of the transfer operator. Next, in Section 5, this result is strengthened to prove the existence of a spectral gap under a certain technical hypothesis concerning the transformation. This hypothesis will be shown to hold for the Rosen continued fractions in Section 6. Finally, we conclude with the statement of the Central Limit Theorem in Section 7.

\section{Iwasawa Continued Fractions}

The results of Sections 3 and 4 hold in general for a wide class of expanding maps of intervals. One suitably broad framework in which we may work in is the theory of \textit{Iwasawa continued fractions}, developed in detail in [3]. This includes the Euclidean algorithm, the even, odd and centered algorithms, continued fractions, alpha continued fractions, Rosen continued fractions, the alpha-Rosen extensions, and their folded variants.

In this framework, $T$ is a map from the interval $I=[a,b]$ to itself, which admits a countable partition into subintervals such that $T$ is of the form $\iota(x)+$(constant) on each subinterval. Here, $\iota$ denotes one of the inversions $\pm 1/x$, $\pm\abs{1/x}$, so that $\abs{T'(x)}=1/x^2$ for all $x\in I$. We denote by $\mathcal{H}$ the set of all inverse branches of $T$, and require that every branch has domain the full interval $I$, except possibly the leftmost and rightmost branches, whose domain should still share an endpoint with $I$.

The continued fraction is called \textit{proper} if $\gamma\coloneqq \max\{\abs{a},\abs{b}\}<1$; we will require this assumption throughout this paper. The only significant algorithm which is not proper is the standard Euclidean algorithm, whose normality have already been studied (see [8]).

\begin{lemma}
For a proper one-dimensional Iwasawa continued fraction, the following bounds hold for all $n$:
\begin{equation*}
\sup_I \frac{1}{\abs{(T^n)'}}\leq \gamma^{2n}\;\;\;\mathrm{and}\;\;\;\sup_I \abs{\left(\frac{1}{\abs{(T^n)'}}\right)'} \leq 2\gamma \frac{1-\gamma^{2n}}{1-\gamma^2}
\end{equation*}
\end{lemma}

\begin{proof}
For the first inequality,
\begin{equation*}
\frac{1}{\abs{(T^n)'(x)}}=\frac{1}{\prod_{i=0}^{n-1} \abs{T'(T^i(x))}}= \prod_{i=0}^{n-1}T^i(x)^2 \leq \gamma^{2n}\,.
\end{equation*}
For the second inequality,
\begin{align*}
\abs{\left(\frac{1}{\abs{(T^n)'}}\right)'(x)} &= \abs{ \left(\prod_{i=0}^{n-1}T^i(x)^2 \right)'} \leq 2\sum_{i=0}^{n-1} \abs{T^i(x)} \prod_{j=0}^{i-1} T'(T^j(x)) \prod_{j=0,\;j\neq i}^{n-1} T^j(x)^2\\
&=2\sum_{i=0}^{n-1} \abs{T^i(x)} \prod_{j=i+1}^{n-1} T^j(x)^2 \leq 2\sum_{i=0}^{n-1} \gamma^{2n-2i-1}=2\gamma \frac{1-\gamma^{2n}}{1-\gamma^2}\,,
\end{align*}
where an empty product is defined to be 1.
\end{proof}

The transfer operator $H$ on $L^1(I)$ is given, by an application of the change-of-variables formula, as follows:

\begin{equation*}
Hf(x)=\sum_{h\in \mathcal{H}} |h'(x)|f\circ h(x)\cdot \chi_{TI_h}(x)
\end{equation*}

Note that $\norm{Hf}_1\leq\int_I\sum_{h\in\mathcal{H}} \abs{h'(x)}\abs{f\circ h(x)} dx=\int_I H\abs{f}(x)dx=\int_I \abs{f(x)}dx$, hence $\norm{H}_1\leq 1$. In fact, there exists an invariant probability measure $f_1m$ for $T$ by [1], where $m$ denotes Lebesgue measure on $I$. Since $f_1$ must satisfy $Hf_1=f_1$, we have $\norm{H}_1=1$.

Letting $\mathcal{H}^n$ denote the set of inverse branches of $T^n$, we more generally have:

\begin{equation*}
H^nf(x)=\sum_{h\in \mathcal{H}^n} |h'(x)|f\circ h(x)\cdot \chi_{T^nI_h}(x)\,.
\end{equation*}

For each $n\in\mathbb{N}$, $\{I_h:h\in\mathcal{H}^n\}$ is a countable partition of $I$ (up to overlapping endpoints). Some simple bookkeeping of subintervals will yield the following.

\begin{prop}
The set $\{T^nI_h:h\in\mathcal{H}^n\}$ is finite, and consists only of the full interval $I$ or intervals sharing an endpoint with $I$.
\end{prop}

In particular, the other endpoint must be one of $T^i(a)$ or $T^i(b)$ for $0\leq i \leq n$, hence the finiteness.

\begin{cor}
For each $n$, there exists an $\epsilon_n>0$ such that $m(T^nI_h)\geq\epsilon_n$ for all $h\in\mathcal{H}^n$.
\end{cor}

\section{Lasota-Yorke Bounds}

In this section, we derive a Lasota-Yorke type inequality using the space of bounded variation. We first state some well-known facts about bounded variation. In order to apply Theorem 1, we work with complex-valued functions.

Recall that a function $f$ on a closed interval $J$ is of \textit{bounded variation} on $J$ if:
\begin{equation*}
\var{J} f \coloneqq \sup_P \sum_{i=1}^{\abs{P}} \abs{f(x_i)-f(x_{i-1})}<\infty,
\end{equation*}
where the supremum runs over all partitions $P=\{\min J=x_0<x_1<\cdots<x_n=\max J\}$. 

The \textit{bounded variation space} $BV_{\ell}(J)$ is the Banach space consisting of all left-continuous, integrable functions of bounded variation on $J$ with the norm,
\begin{equation*}
\norm{f}_{BV}\coloneqq \int_J \abs{f(x)}dx+\var{J} f
\end{equation*}

Taking each element of $BV_{\ell}(J)$ as its equivalence class modulo null sets, we may view the bounded variation space as a subspace $BV(J)$ of $L^1(J)$, with the variation now being taken as the infimum over all equivalent functions (see for example [6]): $\norm{f}_{BV}= \norm{f}_{L^1(J)}+\inf_{g=f\;\mathrm{a.e.}}\var{J} g$.

For $f,g\in BV(J)$, we have the following basic inequalities:
\begin{align*}
\var{J}fg&\leq\sup_J \abs{f}\;\var{J} g+\sup_J \abs{g}\;\var{J} f\,, \\
\var{J}fg&\leq\sup_J \abs{g}\;\var{J} f+\sup_J \abs{g'}\int_J \abs{f(x)}dx \;\;\;\;\mathrm{if}\;g\in C^1(J)\,, \\
\sup_J \abs{f}&\leq \var{J} f+\frac{1}{m(J)}\int_J \abs{f(x)}dx
,.
\end{align*}

In addition, if $J\subseteq I$ and $J$ shares an endpoint with $I$,
$\var{I} f\chi_J\leq\var{J} f+\sup_J \abs{f}$.

We now evaluate $\var{I}H^nf$:

\begin{align}
\var{I}H^nf&=\var{I} \left(\sum_{h\in\mathcal{H}^n} \abs{h'} f\circ h \cdot \chi_{T^nI_h}\right) \leq \sum_{h\in\mathcal{H}^n} \var{I}\left( \abs{h'} f\circ h \cdot \chi_{T^nI_h} \right)\nonumber\\
&\leq \sum_{h\in\mathcal{H}^n} \var{T^nI_h}\left( \abs{h'} f\circ h\right) + \sum_{h\in\mathcal{H}^n} \sup_{T^nI_h}\abs{\abs{h'} f\circ h}\nonumber\\
&= \sum_{h\in\mathcal{H}^n} \var{T^nI_h}\left(\frac{f\circ h}{\abs{(T^n)' \circ h}}\right) + \sum_{h\in\mathcal{H}^n} \sup_{T^nI_h}\abs{\frac{f\circ h}{(T^n)' \circ h}}\nonumber \\
&\leq 2\sum_{h\in\mathcal{H}^n} \var{T^nI_h}\left(\frac{f\circ h}{\abs{(T^n)' \circ h}}\right) + \sum_{h\in\mathcal{H}^n} \frac{1}{m(T^nI_h)}\int_{T^nI_h} \abs{\frac{f\circ h}{(T^n)' \circ h}} dx\nonumber \\
&= 2\sum_{h\in\mathcal{H}^n} \var{I_h}\left(\frac{f}{\abs{(T^n)'}}\right) + \sum_{h\in\mathcal{H}^n} \frac{1}{m(T^nI_h)}\int_{I_h} \abs{f(x)} dx\nonumber \\
&\leq 2\sum_{h\in\mathcal{H}^n} \sup_I\frac{1}{\abs{(T^n)'}} \var{I_h} f + \sum_{h\in\mathcal{H}^n} \left(\sup_I \abs{\left(\frac{1}{\abs{(T^n)'}}\right)'} + \frac{1}{m(T^nI_h)} \right) \int_{I_h}\abs{f(x)} dx\nonumber \\
&\leq 2\gamma^{2n}\,\var{I} f + \left(2\gamma \frac{1-\gamma^{2n}}{1-\gamma^2} + \frac{1}{\epsilon_n}\right) \int_I \abs{f(x)} dx \label{lsb}
\end{align}

Thus, fixing $k\in\mathbb{N}$ large so that $0<\rho\coloneqq2\gamma^{2k}<1$, we have proven the Lasota-Yorke inequality:

\begin{equation*}
\var{I} H^kf\leq\rho\,\var{I}f+M_0\norm{f}_1
\end{equation*}

for some large $M_0$. Since $\norm{H}_1=1$, it immediately follows that

\begin{equation}\label{lasota}
\norm{H^kf}_{BV}\leq\rho\norm{f}_{BV}+M\norm{f}_1\,
\end{equation}
where $M=M_0+1-\rho$.

\section{The Ionescu-Tulcea and Marinescu Theorem}

To obtain quasi-compactness of the operator $H$, we refer to the following theorem of Ionescu-Tulcea and Marinescu [4].

\begin{thm}
Let $(B,\abs{\cdot})$ and $(L,\norm{\cdot})$ be two complex Banach spaces with $L\subset B$. Let $U$ be an operator from $L$ to $L$, bounded with respect to both $\abs{\cdot}$ and $\norm{\cdot}$, which satisfies the additional conditions:

(a) if $f_n\in L$, $f\in B$ with $\abs{f_n-f}\rightarrow 0$ and $\norm{f_n}\leq M$ for all $n$, then $f\in L$ and $\norm{f}\leq M$;

(b) $\sup_{n} \abs{U^n}_L<\infty$;

(c) there exist $k\geq 1$, $0<\rho<1$, and $M$ such that $\norm{U^kf}\leq \rho\norm{f}+M\norm{f}$;

(d) for any bounded subset $L'$ of $(L,\norm{\cdot})$, $U^kL'$ has compact closure in $(B,\abs{\cdot})$.

Then the set $G$ of eigenvalues $\lambda$ of $U$ of modulus 1 is finite, the corresponding eigenspaces $E_{\lambda}$ are finite-dimensional, and $U^n$ can be represented by bounded linear operators $U_{\lambda}$ and $V$ as:

\begin{equation}
U^n=\sum_{\lambda\in G} \lambda^n U_{\lambda}+V^n
\end{equation}
such that
\begin{equation}
U_{\lambda}^2=U_{\lambda},\;\;U_{\lambda}U_{\lambda'}=0\;\;\mathrm{if}\;\; \lambda'\neq \lambda,\;\;U_{\lambda}V=VU_{\lambda}=0,\;\; U_{\lambda}L=E_{\lambda},
\end{equation}
and $V$ has spectral radius strictly less than 1.
\end{thm}

We claim that by taking the two spaces to be $(L^1(I),\norm{\cdot}_1)$ and $(BV(I),\norm{\cdot}_{BV})$ and putting $U=H$, the four conditions (a) to (d) are satisfied and $H$ admits a spectral decomposition of the above form.

\bigskip
\textit{Proof of Claim.}
(a) follows from the fact that the ball $A=\{f\in BV(I): \norm{f}_{BV}\leq M\}$ is $L^1$-compact; see [7]. Similarly, since the image $H^kL'$ of a bounded subset $L'$ of $BV(I)$ is bounded w.r.t. $\norm{\cdot}_{BV}$, it has $L^1$-compact closure, hence (d). For (b), $\sup_{n} \norm{H^n}_1\leq \sup_{n} \norm{H}_1^n\leq 1$. The Lasota-Yorke bound (c) was shown in the preceding section. \qed

\section{Spectral Gap}
Recall that the operator $H$ preserves the invariant density $f_1$ of the transformation $T$, which is of bounded variation; thus 1 is an eigenvalue of $H$. In this section, we prove that 1 is a simple eigenvalue, and there are no other complex eigenvalues of modulus 1. Together with quasi-compactness, this implies that $H$ has spectral gap. We will show this under the following assumption $(*)$ on $T$, which we verify for the Rosen continued fractions in Section 6.

\bigskip
\textbf{Condition $(*)$.} For any subintervals $J,K$ of $I$, there exists $N\geq 0$ such that $m(T^{N}(I)\cap T^{N}(J))>0$.

\begin{lemma}
There exist positive constants $C$, $D$ such that for all $n\geq 0$ and $0\leq r<k$,

\begin{align*}
\norm{H^{nk}f}_{BV}\leq \rho^n\norm{f}_{BV}+C\norm{f}_1,\;\;
\norm{H^rf}_{BV}\leq D\norm{f}_{BV}.
\end{align*}
\end{lemma}

\begin{proof}
By repeated application of the Lasota-Yorke bound,
\begin{align*}
\norm{H^{nk}f}_{BV}&\leq \rho\norm{H^{(n-1)k}f}_{BV}+M\norm{f}_1\\
&\leq \rho\left(\rho\norm{H^{(n-2)k}f}_{BV}+M\norm{f}_1 \right)+M\norm{f}_1\\
&\leq\cdots\leq \rho^n\norm{f}_{BV}+\frac{M}{1-\rho}\norm{f}_1
\end{align*}
In addition, by (\ref{lsb}),
\begin{equation*}
\norm{H^rf}_{BV}\leq 2\gamma^{2r}\,\var{I} f + \left(2\gamma \frac{1-\gamma^{2r}}{1-\gamma^2} +\frac{1}{\epsilon_r}\right)\norm{f}_1 \leq \left(2+\frac{2\gamma}{1-\gamma^2}+\max_{0\leq j<k} \frac{1}{\epsilon_j} \right) \norm{f}_{BV}.
\end{equation*}
\end{proof}

Since $H$ preserves the integral over $I$, if $\lambda\neq 1$, any corresponding eigenfunction must have integral zero. The following proposition allows us to control the behaviour of $H$ on such eigenspaces.

\begin{prop}
	For $f\in BV$ with $\int_I f dm=0$, $H^n f$ converges to 0 in $L^1$-sense as $n\rightarrow \infty$.
\end{prop}

\begin{proof}
It suffices to prove this statement for real-valued $f$. Suppose $\norm{f}_{BV}\leq 1$. Any positive integer $m$ can be expressed as $m=nk+r$ with $n\geq 0$, $0\leq r<k$. Then, by the inequalities of Lemma 2,
\begin{equation*}
\norm{H^mf}_{BV}=\norm{H^{nk+r}f}_{BV}\leq \rho^{n}\norm{H^rf}_{BV}+C\norm{H^rf}_1\leq(C+D)\norm{f}_{BV}.
\end{equation*}
Since the set $\left\{f:\norm{f}_{BV}\leq C+D\right\}$ is $L^1$-compact, there exists a convergent (in $L^1$-sense) subsequence ${H^{n_s}f}$ whose limit is denoted by $f_0\in BV(I)$.

Now we show that $\norm{H^nf_0}_{1}=\norm{f_0}_1$ for all positive integers $n$. Fix $\epsilon>0$. For all large $s$, $\norm{H^{n_s}f-f_0}_1<\epsilon$, so that $\norm{H^nf_0}_1\geq \norm{H^{n+n_s}f}_1-\epsilon$.
Fixing $t>s$ such that $n+n_s<n_t$, 
\begin{equation*}
\norm{H^{n+n_s}f}_1\geq\norm{H^{n_t}f}_1 \geq\norm{f_0}_{1}-\epsilon.
\end{equation*}
Thus $\norm{H^nf_0}_1\geq \norm{f_0}_1-2\epsilon$; letting $\epsilon$ tend to 0, we have $\norm{H^nf_0}_1\geq \norm{f_0}_1$ for all $n$. The opposite inequality is trivial.

Since $\int_I H^{n_s}fdm=\int_I fdm=0$, we have $\int_I f_0 dm=0$. Suppose $f_0\neq 0$ a.e. Viewing $f_0$ as a left-continuous function in $BV_{\ell}(I)$, there exists $a<\alpha,\beta<b$ such that $f_0(\alpha)>0, f_0(\beta)<0$, which implies $f_0$ is strictly positive (resp. negative) on $[\alpha-\epsilon,\alpha]$ (resp. $[\beta-\epsilon,\beta]$), for some $\epsilon>0$.

Note that if $f$ is positive on $J\subseteq I$, then $Hf$ is strictly positive on $T(J)$. By Condition $(*)$, there exists $N$ such that $A=T^N([\alpha-\epsilon,\alpha]) \cap T^N([\beta-\epsilon,\beta])$ has positive measure.

Decompose $f_0$ into its positive and negative parts, $f_0=f_+-f_-$. Then,
\begin{equation*}
\norm{H^Nf_+}_1+\norm{H^Nf_-}_1 =\norm{f_0}_1 =\norm{H^Nf_0}_1=\norm{H^N f_+-H^{N}f_-}_1.
\end{equation*}

However, both $H^Nf_+$ and $H^Nf_-$ are strictly positive on $A$, so the triangle inequality must be strict; a contradiction. We conclude that $f_0=0$ a.e.
\end{proof}

We are now in a position to prove the existence of a spectral gap for $H$.

\begin{thm}
The eigenvalue 1 is simple, and there are no other eigenvalues of modulus 1.
\end{thm}

\begin{proof}
Suppose that $f_2$ is another eigenfunction corresponding to 1. Let $g=f_1\int_I f_2dm-f_2$. Then $Hg=g$ and $\int_I gdm=0$. By the preceding proposition, $\norm{g}_1=\norm{H^ng}_1\rightarrow 0$, so $g=0$ a.e. Thus $f_2$ must be a constant multiple of $f_1$.

It remains to show that there is no other eigenvalue on the unit circle. Suppose that $\lambda\neq 1$, $\abs{\lambda}=1$ is an eigenvalue of $H$, with a corresponding eigenfunction $g$. Since $$\int_I gdm=\int_I Hgdm=\lambda\int_I gdm,$$
we have $\int_Igdm=0$. Therefore $H^ng=\lambda^ng$ must converge to 0 by the preceding proposition, contradicting $\abs{\lambda}=1$.
\end{proof}

\section{Condition $(*)$ for the Rosen Map}
In this section, we demonstrate Condition $(*)$ for the Rosen continued fractions. In fact, it follows trivially from the theorem below.

\begin{thm}
For any subinterval $[c,d]\subseteq [-\lambda_{q}/2,\lambda_{q}/2]$, $T_q^N ([c,d])=[-\lambda_{q}/2,\lambda_{q}/2]$ a.e. for large enough $N$.
\end{thm}

\begin{proof}
We prove the theorem when $q\geq 5$; the cases $q=3,4$ are similarly obtained, but with fewer cases to consider.

Let $\sigma_{q}=\lambda_q/2$. The ranges $I_j$ of the inverse branches $h_j$ are, with appropriate indexing:
\begin{align*}
I_1 = [\frac{1}{3\sigma_{q}},2\sigma_{q}],\;\;\;
I_{2n+1} = [\frac{1}{(2n+3)\sigma_q},\frac{1}{(2n+1)\sigma_q}], \;\;\;I_{2n} = -\,I_{2n-1}\;\;(n\geq 1)\,.
\end{align*}
Note that $\frac{1}{\sqrt{2}}\in I_1$ and $T_{q}(\frac{1}{\sqrt{2}})=\sqrt{2}-\lambda_{q}<0$, so we may choose $\epsilon>0$ with $\frac{1}{\sqrt{2}}-\epsilon \in I_1$ and $T_{q}(\frac{1}{\sqrt{2}}-\epsilon)<0$.

\bigskip
We shall call an interval which does not satisfy the conclusion of the theorem a `bad interval.' Suppose a bad interval $[c,d]$ exists. Clearly $0\notin[c,d]$, for otherwise we may take $N=1$.

If $0<c<d$, let $c\in I_{2n+1}$, $d\in I_{2m+1}$ $(n\geq m)$. If $n>m+1$, $T_q[c,d]$ must be the full interval, a contradiction. Thus $n=m$ or $m+1$. Similarly, if $c<d<0$, let $c\in I_{2n}$, $d\in I_{2m}$, where $n=m$ or $m-1$.

For these four types of bad intervals, the image $T_{q}[c,d]$ are given below.

\begin{align*}
\mathrm{type\;(a):} &\quad\quad T_{q}[c,d]=[T_q(d),T_q(c)]\quad (0<c<d,\;n=m) \\
\mathrm{type\;(b):} &\quad\quad T_{q}[c,d]=[T_q(c),T_q(d)]\quad (c<d<0,\;n=m) \\
\mathrm{type\;(c):} &\quad\quad T_{q}[c,d]=[-\sigma_q,T_q(c)]\cup[T_q(d),\sigma_q]\quad (0<c<d,\;n=m+1) \\
\mathrm{type\;(d):} &\quad\quad T_{q}[c,d]=[-\sigma_q,T_q(d)]\cup[T_q(c),\sigma_q]\quad (c<d<0,\;n=m-1)
\end{align*}
\bigskip
Moreover, for type (c), if $d\geq\frac{1}{\sqrt{2}}-\epsilon$, $T_q(d)<0$ so $0\in T_q[c,d]$. Therefore $T_{q}^{2}[c,d]=[-\lambda_{q}/2,\lambda_{q}/2]$, a contradiction. This implies $0<c<d<\frac{1}{\sqrt{2}}-\epsilon$.
Similarly, $-\frac{1}{\sqrt{2}}+\epsilon<c<d<0$ holds for type (d). Now, define a sequence of bad intervals ${[p_n,q_n]}$ by $[p_0,q_0]=[c,d]$, and
\begin{align*}
[p_{n+1},q_{n+1}]=&[T_q(q_n),T_q(p_n)] \quad\quad\quad\quad\quad\quad \mathrm{if}\;[p_n,q_n] \; \mathrm{is}\;\mathrm{of}\;\mathrm{type\;(a)}\\
&[T_q(p_n),T_q(q_n)] \quad\quad\quad\quad\quad\quad \mathrm{if}\;[p_n,q_n]\;\mathrm{is}\;\mathrm{of}\;\mathrm{type\;(b)}\\
&[-T_q(c),\sigma_q]\cup[T_q(d),\sigma_q] \,\quad\quad \mathrm{if}\;[p_n,q_n]\;\mathrm{is}\;\mathrm{of}\;\mathrm{type\;(c)}\\
&[-T_q(d),\sigma_q]\cup[T_q(c),\sigma_q] \,\quad\quad \mathrm{if}\;[p_n,q_n]\;\mathrm{is}\;\mathrm{of}\;\mathrm{type\;(d)}
\end{align*}
This sequence is well defined since $T_q(I)=T_q(-I)$ for all $I\subset[-\lambda_{q}/2,\lambda_{q}/2]$.

\bigskip
We now show that there exists an $\eta>1$ so that $m([p_{n+1},q_{n+1}])>\eta\cdot m([p_{n},q_{n}])$ for all $n$, which yields the desired contradiction as $n\rightarrow \infty$.

Since $\abs{T_q'(x)}=1/x^2$, for $[p_n,q_n]$ of type (a) or (b), $$m([p_{n+1},q_{n+1}])\geq\cfrac{1}{\sigma_q^2}m([p_n,q_n]).$$
For $[p_n,q_n]$ of type (c) or (d), since $\max\{|p_n|,|q_n|\}<\frac{1}{\sqrt{2}}-\epsilon$, 
$$|T_{q}'(x)|\geq\frac{1}{(1/\sqrt{2}-\epsilon)^2}\;\;\;\forall x \in [p_n,q_n]$$
Therefore, for type (c),
$$m([p_{n+1},q_{n+1}])\geq \frac{1}{2} \left(m([-\sigma_q,T_q(c)])+m([T_q(d),\sigma_q])\right) \geq \frac{1}{2}\frac{1}{(1/\sqrt{2}-\epsilon)^2}m([p_n,q_n])\,.$$
And the same holds for type (d).
\end{proof}

\section{Central Limit Theorem}
For the statement of the Central Limit Theorem for operators with a spectral gap, we refer to [2]. For $f\in BV(I)$, let $S_{N}f$ denote the Birkhoff sum $\sum_{k=0}^{N-1} f\circ T^{k}$, and let $\mu=f_{1}m$ be the invariant measure of $T$ on $I$.

Following Broise, we define a condition on functions $f$ of bounded variation, which is equivalent to a degeneracy $\sigma=0$ of the variance of the distribution of Birkhoff sums.

\bigskip
\textbf{Condition (H).} There exists $u\in L^{2}(\mu)$ with $f=u-u\circ T+ \int_I fd\mu$.

\begin{thm}
	For $f\in BV(I)$ that does \textup{not} satisfy Condition (H), there exists $\sigma>0$ so that:
$$\lim_{N\rightarrow\infty} \mu\left[\frac{S_{N}f-N\int_I fd\mu}{\sigma\sqrt{N}}\leq v\right]=\frac{1}{\sqrt{2\pi}}\int_{-\infty}^{v}e^{-t^2/2}dt\quad \mathrm{for\;all\;} v.$$
\end{thm}

In particular, by taking $f$ to be constant on each subinterval $I_h$, we may interpret $f$ as a cost function on the set of inverse branches $\mathcal{H}$. (The orbits which hit the endpoints of $I_h$'s have measure zero and may be ignored.) The condition $f\in BV(I)$ is satisfied, for example, by monotonic and bounded costs with respect to the integer values $a_i$ in the expansion (\ref{expansion}). In this case, Condition (H) is satisfied only if $f$ is constant on all branches excluding the leftmost one, as all other branches admit fixed points.

\begin{cor}
For any nonconstant, monotonic, and bounded cost on the digits $\abs{a_i}$, the total cost of an execution of the Rosen algorithm, with input randomly chosen from the invariant distribution $\mu$, will be asymptotically Gaussian.
\end{cor}

Further analyses into Condition (H) for general $f$ are also given in [2]. For example, suppose $f$ is piecewise continuous. We may add a constant to $f$ so that $\int_I f d\mu=0$. Now, if a point $x\in I$ with period $k$ can be found whose orbit does not meet the countable set consisting of the orbits of discontinuities of $f$ or the endpoints of the subintervals $I_h$, and the Birkhoff sum $S_kf(x)$ is nonzero, then $f$ does not satisfy Condition (H).

\bibliographystyle{unsrt}
\bibliography{a}

\end{document}